\documentclass[12 pt,english,oneside]{amsart}
\usepackage[top=1.26in, bottom=1.26in, left=1.33in, right=1.33in]{geometry}
\usepackage{verbatim}
\usepackage{amsfonts}
\usepackage{amsmath}
\usepackage{amssymb}
\usepackage{amsthm}
\usepackage{amscd}
\usepackage{dsfont}
\usepackage{fancyhdr}

\newcommand{\g}{\mathfrak{g}}

\newcommand{\D}{\displaystyle}
\newcommand{\ra}{\rightarrow}

\newcommand{\ve}{\varepsilon}
\newcommand{\vp}{\varphi}

\newcommand{\bino}[2]{\left[\genfrac{}{}{0pt}{}{#1}{#2}\right]}
\newcommand{\ts}{\otimes}
\newcommand{\rr}{\mathcal{R}}

\newtheorem{definition}{Definition}
\newtheorem{proposition}{Proposition}
\newtheorem{theorem}{Theorem}
\newtheorem{lemma}{Lemma}
\newtheorem{corollary}{Corollary}
\newtheorem{remark}{Remark}

\title{Quantum groups, q-Boson algebras and quantized Weyl algebras}
\author{Xin Fang}
\address{
Universit\'{e} Paris 7, Institut de Math\'{e}matiques de Jussieu, Th\'{e}orie des groupes et des
repr\'{e}sentations, Case 7012, 2 place Jussieu, 75251 Paris Cedex 05, France.}
\email{fang@math.jussieu.fr}

\begin{document}

\maketitle
\begin{abstract}
We give a unified construction of quantum groups, q-Boson algebras and quantized Weyl algebras and
an action of quantum groups on quantized Weyl algebras. This enables us to give a conceptual proof
of the semi-simplicity of the category $\mathcal{O}(B_q)$ introduced by T.Nakashima and the classification of all simple objects
in it.
\end{abstract}

\section{Introduction}
In his article \cite{Kas91}, M.Kashiwara defined crystal bases for quantized enveloping algebras. To
show the existence of such bases for the strictly negative part $U_q^-(\g)$
of quantized enveloping algebras, he constructed an associative algebra generated by operators on $U_q^-(\g)$,
which is a q-analogue of boson. In fact, this algebra is a quantized version of the usual Weyl algebra and with the help of such algebra, he
proved that $U_q^-(\g)$, viewed as a module over this "quantized Weyl algebra", is simple. Moreover, he
affirmed without proof
that imposing a finiteness condition on modules over "quantized Weyl algebra" will lead to
semi-simplicity results.\\
\indent
Later, in his article \cite{Nak94}, T.Nakashima defined the so called "q-Boson algebra" $B_q(\g)$, an
extension of quantized Weyl algebra $W_q(\g)$ by a torus, and studied these algebras.
Finally, in \cite{Nak04}, he proved the semi-simplicity of
$\mathcal{O}(B_q)$, the category of modules over $B_q(\g)$ with some finiteness conditions, where the main tool is an "extremal projector" also defined in \cite{Nak04}. 
But we should point out that the proof in \cite{Nak04} depends on the "Casimir-like" element of a pairing;
to get the desired properties, the author has to use a large quantity of computation, see for example
\cite{Tan92}, \cite{Nak94} and \cite{Nak04}.\\
\indent
In this article, we will construct quantized enveloping algebras(quantum groups), q-Boson algebras
and quantized Weyl algebras in a unified method and give an action of quantum groups on quantized
Weyl algebras by the Schr\"{o}dinger representation. This enables us to give another construction
of quantized Weyl algebras with the help of the braiding in some Yetter-Drinfel'd module category.
With this construction, we can obtain a structural result for all $W_q(\g)$-modules with
a natural finiteness condition, which will lead directly to the semi-simplicity of $\mathcal{O}(B_q)$
and the classification of all simple objects in it. Moreover, the proof we give here is more conceptual: it means that
the structure of category $\mathcal{O}(B_q)$ depends heavily on the intrinsic duality of $B_q(\g)$.
As a byproduct, we prove the semi-simplicity of $W_q(\g)$-modules with a finiteness condition
and classify all simple modules of this type.\\
\indent
This work is inspired by an observation in the finite dimensional case: once we have a nondegenerate
pairing between two Hopf algebras, we may form the smash product of them, where the "module algebra type"
action is given by this pairing. If we have a finite dimensional module over this smash product,
from the duality, we will obtain simultaneously a module and a comodule structure, and the construction
of smash product is exactly the compatibility condition of the module and comodule structure to yield a Hopf module. As showed in \cite{Swe94}, all Hopf modules are trivial,
it is to say, a free module over the original Hopf algebra, and blocks are parameterized by a vector space
called "coinvariants".\\
\indent
We would like to generalize this observation to a more general case, for example, quantized Weyl algebras or q-Boson algebras.
But unfortunately, it does not work well because the action of torus part is not locally nilpotent. Our main idea for overcoming this difficulty is to hide the "torus part" behind the construction with the help
of a braiding originated in a quantum group action. This is the main reason for our use of the technical language
of Yetter-Drinfel'd modules and braided Hopf algebras.\\
\indent
Now we want to be more precise: for any module $M$ in $\mathcal{O}(B_q)$,
it is possible to restrict it to the quantized Weyl algebra to obtain a
$W_q(\g)$-module with finiteness condition. In section \ref{wqbraid}, we realize $W_q(\g)$ as an algebra obtained
from its negative and positive parts with a braiding, this enables us to get a module and comodule
structure on $M$. Unfortunately, for a $W_q(\g)$-module $M$, we will not have the compatibility condition, but it is not
too far away: they are also compatible in the sense of braiding in this case; we may still prove a trivialization result, which gives out the structure theorem of all
$W_q(\g)$-modules with finiteness condition and will lead easily to the structure theory of category $\mathcal{O}(B_q)$.\\
\indent
As in the proof for the structural theorem of Hopf modules, there exists a projection from the module to its coinvariants, which will be shown to be exactly the "extremal projector" in \cite{Nak04} and the projection given
in \cite{Kas91} in the $\mathfrak{sl}_2$ case. This explains the "extremal projector" in a more natural way.\\
\indent
The constitution of this article is as follows. In Section \ref{part1}, we recall some notions in Hopf
algebras and give out an action of quantum doubles on Heisenberg doubles with the help of Schr\"{o}dinger
representations. In Section \ref{part2}, we construct quantum groups and q-Boson algebras concretely and
calculate the action between them in the case of $\mathfrak{sl}_2$. In Section \ref{part3}, we
construct quantized Weyl algebras from the braiding in Yetter-Drinfel'd category and prove the main
theorem for the structure of $\mathcal{O}(B_q)$, at last, we compare our projection with those
defined in \cite{Kas91} and \cite{Nak04}.\\
\indent
At last, we should remark that in the preparation of this article, the preprint of A.M.Semikhatov \cite{Sem10} came into our sight, he
got essentially same results as in the Section \ref{part1} of this article, though with a different objective and point of view.\\
\\
\noindent
\textbf{Acknowledgements.}
I would like to thank my advisor Marc Rosso for suggesting me this problem,
for his guidance and encouragement. I want to thank Can Zhang for her constant support. Treasurable remarks from the referee of introducing me the article \cite{Mas09}, Miyashita-Ulbrich action and an excellent reformulation, could never be overestimated. This work is partially supported by Sino-France Collaboration Grant 34000-3275100 from Sun Yat-sen University.

\section{Hopf pairings and double constructions}\label{part1}
From now on, suppose that we are working in the complex field $\mathbb{C}$. Results in this section hold for
any field with characteristic $0$. All tensor products are over $\mathbb{C}$ if not specified otherwise.
\subsection{Yetter-Drinfel'd modules}
Let $H$ be a Hopf algebra. A vector space $V$ is called a (left) $H$-Yetter-Drinfel'd module if it is simultaneously
an $H$-module and an $H$-comodule which satisfy the Yetter-Drinfel'd compatibility condition: for any
$h\in H$ and $v\in V$,
$$\sum h_{(1)}v_{(-1)}\ts h_{(2)}.v_{(0)}=\sum (h_{(1)}.v)_{(-1)}h_{(2)}\ts (h_{(1)}.v)_{(0)},$$
where $\Delta(h)=\sum h_{(1)}\ts h_{(2)}$ and $\rho(v)=\sum v_{(-1)}\ts v_{(0)}$ are Sweedler
notations for coproduct and comodule structure maps.\\
\indent
Morphisms between two $H$-Yetter-Drinfel'd modules are linear maps preserving $H$-module
and $H$-comodule structures.\\
\indent
We denote the category of $H$-Yetter-Drinfel'd modules by ${}_H^H\mathcal{YD}$, this is a tensor category.\\
\indent
The advantage of Yetter-Drinfel'd module is: for $V,W\in {}_H^H\mathcal{YD}$, there exists a braiding
$\sigma:V\ts W\ra W\ts V$, given by $\sigma(v\ts w)=\sum v_{(-1)}.w\ts v_{(0)}$. If both $V$ and
$W$ are $H$-module algebras, $V\ts W$ will have an algebra structure
if we use $\sigma$ instead of the usual flip.

\subsection{Braided Hopf algebras in ${}_H^H\mathcal{YD}$}
In \cite{Rad85}, D.Radford constructed the biproduct of two Hopf algebras when there exists an action
and coaction between them and obtained the necessary and sufficient conditions for the existence of a Hopf
algebra structure on this biproduct. See Theorem 1 and Proposition 2 in \cite{Rad85}.\\
\indent
Once the language of Yetter-Drinfel'd module has been adopted, conditions in \cite{Rad85} can be easily rewritten.
\begin{definition}[\cite{AS94}, Section 1.3]
A braided Hopf algebra in the category ${}_H^H\mathcal{YD}$ is a collection $(A,m,\eta,\Delta,\ve,S)$ such that:\\
(1). $(A,m,\eta)$ is an algebra in ${}_H^H\mathcal{YD}$; $(A,\Delta,\ve)$ is a coalgebra in
${}_H^H\mathcal{YD}$. It is to say, $m,\eta,\Delta,\ve$ are morphisms in ${}_H^H\mathcal{YD}$;\\
(2). $\Delta:A\ra A\underline{\ts} A$ is a morphism of algebra. The notation $\underline{\ts}$ stands for the tensor
product of two Yetter-Drinfel'd module algebras where we use the braiding in ${}_H^H\mathcal{YD}$ instead of the usual flip;\\
(3). $\ve:A\ra \mathbb{C}$, $\eta:\mathbb{C}\ra A$ are algebra morphisms;\\
(4). $S$ is the convolution inverse of $Id_A\in End(A)$.
\end{definition}
\begin{remark} (1). Once a braided Hopf algebra $A$ has been given, we can form the tensor product
$A\ts H$, it yields a Hopf algebra structure, as shown in \cite{Rad85}.\\
(2). An important example here is the construction of the "positive part" of a quantized enveloping algebra
as a twist of a braided Hopf algebra with primitive coproduct by a commutative group algebra.\\
(3). For a general construction in the framework of Hopf algebras with a projection, see \cite{AS94},
section 1.5.
\end{remark}

\subsection{Braided Hopf modules}

Let $B$ be a braided Hopf algebra in some Yetter-Drinfel'd module category. For a left braided $B$-Hopf module $M$, we mean a left $B$-module and a left $B$-comodule satisfying compatibility condition as follows:
$$\rho\circ l=(m\ts l)\circ (id\ts\sigma\ts id)\circ (\Delta\ts\rho):B\ts M\ra B\ts M,$$
where $m$ is the multiplication in $B$, $l:B\ts M\ra M$ is the module structure map, $\rho:M\ra B\ts M$ is the comodule structure map and $\sigma$ is the braiding in the fixed Yetter-Drinfel'd category.\\
\noindent
\textbf{Example.} Let $V$ be a vector space over $\mathbb{C}$. Then $B\ts V$ admits a $B$-braided Hopf module structure given by: for $b,b'\in B$ and $v\in V$,
$$b'.(b\ts v)=b'b\ts v,\ \ \rho(b\ts v)=\sum b_{(1)}\ts b_{(2)}\ts v\in B\ts (B\ts V).$$
\qed\\
\indent
We let ${}_B^B\mathcal{M}$ denote the category of left braided $B$-Hopf modules. The following proposition gives the triviality of such kind of modules.
\begin{proposition}\label{triviality}
Let $M\in{}_B^B\mathcal{M}$ be a braided Hopf module, $\rho:M\ra B\ts M$ be the structural map, $M^{co\rho}=\{m\in M|\ \rho(m)=1\ts m\}$ be the set of coinvariants. Then there exists an isomorphism of braided $B$-Hopf modules:
$$M\cong B\ts M^{co\rho},$$
where the right hand side adopts the trivial Hopf module structure. Moreover, maps in two directions are given by:
$$M\ra B\ts M^{co\rho},\ \ m\mapsto\sum m_{(-1)}\ts P(m_{(0)}),$$
$$B\ts M^{co\rho}\ra M,\ \ b\ts m\mapsto bm,$$
where $m\in M$, $b\in B$ and $P:M\ra M^{co\rho}$ is defined by: $P(m)=\sum S(m_{(-1)})m_{(0)}$.
\end{proposition}
The proof for the triviality of Hopf modules given in \cite{Swe94} can be adopted to the braided case.
\begin{remark}
Proposition \ref{triviality} can be translated into the categorical language, which says that there exists an equivalence of category ${}_B^B\mathcal{M}\sim Vect$, where $Vect$ is the category of vector spaces, given by $M\mapsto M^{co\rho}$ and $V\mapsto B\ts V$ for $M\in{}_B^B\mathcal{M}$ and $V\in Vect$.
\end{remark}

\subsection{Generalized Hopf pairings}
Generalized Hopf pairings give dualities between Hopf algebras.\\
\indent
Let $A$ and $B$ be two Hopf algebras with invertible antipodes. A
generalized Hopf pairing between $A$ and $B$ is a bilinear form
$\vp:A\times B\ra\mathbb{C}$ such that:\\
(1). for any $a\in A$, $b,b'\in B$,
$\vp(a,bb')=\sum\vp(a_{(1)},b)\vp(a_{(2)},b')$;\\
(2). for any $a,a'\in A$, $b\in B$,
$\vp(aa',b)=\sum\vp(a,b_{(2)})\vp(a',b_{(1)})$;\\
(3). for any $a\in A$, $b\in B$,
$\vp(a,1)=\ve(a),\ \ \vp(1,b)=\ve(b)$.
\begin{remark}
From the uniqueness of the antipode and conditions (1)-(3) above, we have: for any $a\in A$, $b\in B$,
$\vp(S(a),b)=\vp(a,S^{-1}(b))$.
\end{remark}

\subsection{Quantum doubles}
Let $A$ and $B$ be two Hopf algebras with invertible antipodes and $\vp$ be
a generalized Hopf pairing between them. The quantum double $D_\vp(A,B)$ is defined by:\\
(1). as a vector space, it is $A\ts B$;\\
(2). as a coalgebra, it is the tensor product of coalgebras $A$ and $B$;\\
(3). as an algebra, the multiplication is given by:
$$(a\ts b)(a'\ts b')=\sum \vp(S^{-1}(a_{(1)}'),b_{(1)})\vp(a_{(3)}',b_{(3)})aa_{(2)}'\ts b_{(2)}b'.$$

\subsection{Schr\"{o}dinger Representations}\label{Sch}
The prototype of Schr\"{o}dinger representation in physics is the momentum group $G$ action on a
position space $M$; this will give out an action of
$\mathbb{C}(M)\rtimes\mathbb{C}(G)$ on $\mathbb{C}(M)$. Details of this view point can be found
in the chapter 6 of \cite{Maj95}.\\
\indent
The definition and proposition in this subsection are essentially in \cite{Maj95}, Example 7.1.8.\\
\indent
The Schr\"{o}dinger representation of $D_\vp(A,B)$ on $A$ is given
by: for $a,x\in A$, $b\in B$,
$$(a\ts 1).x=\sum a_{(1)}xS(a_{(2)}),$$
$$(1\ts b).x=\sum \vp(x_{(1)},S(b))x_{(2)}.$$
The Schr\"{o}dinger representation of $D_\vp(A,B)$ on $B$ is given
by: for $a\in A$, $b,y\in B$,
$$(a\ts 1).y=\sum \vp(a,y_{(1)})y_{(2)},$$
$$(1\ts b).y=\sum b_{(1)}yS(b_{(2)}).$$
So
$$(a\ts b).x=\sum \vp(x_{(1)},S(b))a_{(1)}x_{(2)}S(a_{(2)}),$$
$$(a\ts b).y=\sum \vp(a,b_{(1)}y_{(1)}S(b_{(4)}))b_{(2)}y_{(2)}S(b_{(3)}).$$
\begin{proposition}[\cite{Maj95}, Example 7.1.8]\label{AB}
With the definition above, both $A$ and $B$ are $D_\vp(A,B)$-module
algebras.
\end{proposition}

\subsection{Heisenberg doubles}
Keep assumptions in previous sections. Now we construct the Heisenberg double
between $A$ and $B$; it is the smash product of them where the module algebra type action of $A$
on $B$ is given by the Hopf pairing. For the background of this double, see \cite{Lu94}.\\
\indent
The Heisenberg double $H_\vp(A,B)$ is an algebra defined as follows:\\
(1). as a vector space, it is $B\ts A$ and we denote the pure tensor by $b\sharp a$;\\
(2). the product is given by: for $a,a'\in A$, $b,b'\in B$,
$$(b\sharp a)(b'\sharp a')=\sum \vp(a_{(1)},b_{(1)}')bb_{(2)}'\sharp a_{(2)}a'.$$
\noindent
\begin{remark} 
In general, $H_\vp(A,B)$ has no Hopf algebra structure.
\end{remark}

\subsection{Quantum double action on Heisenberg double}
We define an action of $D_\vp(A,B)$ on $H_\vp(A,B)$ as follows: for
$a,a'\in A$, $b,b'\in B$,
$$(a\ts b).(b'\sharp a')=\sum (a_{(1)}\ts b_{(1)}).b'\sharp (a_{(2)}\ts b_{(2)}).a',$$
this is a diagonal type action. Moreover, we have the
following result:
\begin{proposition}\label{modulealgebra}
With this action, $H_\vp(A,B)$ is a $D_\vp(A,B)$-module algebra.
\end{proposition}

To be more precise, the above action can be written as:
$$(a\ts b).(b'\sharp a')=\sum \vp(a_{(1)},b_{(1)}b_{(1)}'S(b_{(4)}))\vp(a_{(1)}',S(b_{(5)}))
b_{(2)}b_{(2)}'S(b_{(3)})\sharp a_{(2)}a_{(2)}'S(a_{(3)}).$$
\begin{remark}
This proposition gives a family of examples for Yang-Baxter algebras; for the
definition and fundamental properties, see \cite{JR09}. Properties of such kind of algebra make it
possible to define a braiding on the tensor product of $H_\vp(A,B)$, which gives an algebra structure on
$H_\vp(A,B)^{\ts n}$. Equivalently, we can translate this braiding in the framework of Yetter-Drinfel'd
modules, which is much more useful for future applications.
\end{remark}
\indent
We define a $D_\vp(A,B)$-comodule structure on both $A$ and $B$ as follows:
$$A\ra D_\vp(A,B)\ts A,\ \ a\mapsto \sum a_{(1)}\ts 1\ts a_{(2)},$$
$$B\ra D_\vp(A,B)\ts B,\ \ b\mapsto \sum 1\ts b_{(1)}\ts b_{(2)}.$$

\begin{proposition}\label{yedrinAB}
With Schr\"{o}dinger representations and comodule structure maps defined above, both $A$ and $B$ are in the category ${}^{D_\vp}_{D_\vp}\mathcal{YD}$.
\end{proposition}
More generally, we have the following result.
\begin{proposition}\label{yetterdrinfeld}
With the comodule structure map defined by:
$$\delta:H_\vp(A,B)\ra D_\vp(A,B)\ts H_\vp(A,B),\ \ b\sharp a\mapsto \sum ((1\ts b_{(1)})(a_{(1)}\ts 1))\ts b_{(2)}\sharp a_{(2)},$$
for $a\in A$, $b\in B$, $H_\vp(A,B)$ is in the category ${}^{D_\vp}_{D_\vp}\mathcal{YD}$.
\end{proposition}

The rest part of this section will be devoted to giving proofs of these propositions using the Miyashita-Ulbrich action. This is recommended by the referee.

\subsection{Twisted product}
Let $H$ be a Hopf algebra over $\mathbb{C}$, $\sigma:H\ts H\ra\mathbb{C}$ be a 2-cocycle which is invertible in $(H\ts H)^*$ (for a definition, see \cite{DT94} or \cite{Lu94}). Then we can form the following two twisted products on $H$.

\begin{definition}
The twisted algebra $H^\sigma$ is defined as follows:\\
(1). as a vector space, it is $H$ itself;\\
(2). for any $x,y\in H$, the product in $H^\sigma$ is given by:
$$x\bullet y=\sum \sigma(x_{(1)},y_{(1)})x_{(2)}y_{(2)}\sigma^{-1}(x_{(3)},y_{(3)}),$$
where $\sigma^{-1}$ is the inverse of $\sigma$ in $(H\ts H)^*$.
\end{definition}

With the original coproduct on $H$, $H^\sigma$ is a Hopf algebra.

\begin{definition}
The twisted algebra ${}_\sigma H$ is defined as follows:\\
(1). as a vector space, it is $H$ itself;\\
(2). for any $x,y\in H$, the product in ${}_\sigma H$ is given by:
$$x\circ y=\sum \sigma(x_{(1)},y_{(1)})x_{(2)}y_{(2)}.$$
\end{definition}

The coproduct on $H$ gives ${}_\sigma H$ a left $H^\sigma$-comodule algebra structure. Moreover, ${}_\sigma H$ is cleft $H^\sigma$-Hopf-Galois extension over $\mathbb{C}$ on the left such that the identity map $\gamma:H^\sigma\ra {}_\sigma H$ is a convolution-invertible $H^\sigma$-comodule morphism (see, for example, Theorem 4.3 in \cite{MS99}). We let $\gamma^{-1}$ denote the convolution-inverse of $\gamma$.\\
\indent
For $x\in H^\sigma$, $y\in {}_\sigma H$,
$$x\rightharpoonup y=\sum \tau(x_{(1)})\circ y\circ \tau^{-1}(x_{(2)})$$
gives the Miyashita-Ulbrich action of $H^\sigma$ on ${}_\sigma H$ (see \cite{Sch96} for the definition).\\
\indent
From Corollary 3.1 in \cite{Sch96}, the Miyashita-Ulbrich action of $H^\sigma$ on ${}_\sigma H$ and the original coaction make ${}_\sigma H$ into an algebra object in the category ${}_{H^\sigma}^{H^\sigma}\mathcal{YD}$.

\subsection{Application to the double construction}
We preserve notations in previous sections.\\
\indent
Let $A,B$ be two Hopf algebras and $H=B\ts A$ be their tensor product. Suppose that there exists a Hopf pairing $\vp$ between $A$ and $B$. Then we can define a 2-cocycle using this pairing: for $a,a'\in A$ and $b,b'\in B$,
$$\sigma:H\ts H\ra\mathbb{C},\ \ \sigma(b\ts a,b'\ts a')=\ve(b)\vp(a,b')\ve(a').$$
Moreover, the inverse of $\sigma$ is given by:
$$\sigma^{-1}:H\ts H\ra\mathbb{C},\ \ \sigma^{-1}(b\ts a,b'\ts a')=\ve(b)\vp(a,S(b'))\ve(a').$$

\begin{proposition}[\cite{DT94},\cite{Lu94}]\label{isom}
(1). There exists an isomorphism of Hopf algebras:
$$D_\vp(A,B)\ra H^\sigma,\ \ a\ts b\mapsto (1\ts a)\bullet (b\ts 1).$$
(2). As algebras, $H_\vp(A,B)={}_\sigma H$.
\end{proposition}

Thus $H_\vp$ is a cleft $D_\vp$-Hopf-Galois extension over $\mathbb{C}$ on the left.

In this case, we compute the Miyashita-Ulbrich action explicitly. Note that $D_\vp$ includes $A$, $B$ as Hopf subalgebras, and that $H_\vp$ includes $A$ (resp., $B$) as a left $A$- (resp., $B$-)comodule subalgebra. It results that the identity map $D_\vp\overset{\sim}{\longrightarrow} H^\sigma\ra{}_\sigma H=H_\vp$, restricted to $A$, $B$, has antipodes of $A$, $B$ as convolution inverses.\\
\indent
Let $a,a'\in A$ and $b,b'\in B$. Then
\begin{eqnarray*}
(a\ts 1)\rightharpoonup (b\ts 1) &=& \sum (1\ts a_{(1)})(b\ts 1)(1\ts S(a_{(2)}))\\
&=& \sum\vp(a,b_{(1)})b_{(2)}\ts 1;
\end{eqnarray*}
\begin{eqnarray*}
(a\ts 1)\rightharpoonup (1\ts a') &=& \sum (1\ts a_{(1)})(1\ts a')(1\ts S(a_{(2)}))\\
&=& 1\ts\sum a_{(1)}a'S(a_{(2)});
\end{eqnarray*}
\begin{eqnarray*}
(1\ts b)\rightharpoonup (1\ts a) &=& \sum (b_{(1)}\ts 1)(1\ts a)(S(b_{(2)}\ts 1)\\
&=& 1\ts \sum\vp(a_{(1)},S(b))a_{(2)}; 
\end{eqnarray*}
\begin{eqnarray*}
(1\ts b)\rightharpoonup (b'\ts 1) &=& \sum (b_{(1)}\ts 1)(b'\ts 1)(S(b_{(2)}\ts 1)\\
&=& \sum b_{(1)}b'S(b_{(2)})\ts 1.
\end{eqnarray*}
This recovers Schr\"{o}dinger representations of $D_\vp(A,B)$ on $A$ and $B$.\\
\indent
Now Proposition \ref{AB}, \ref{modulealgebra} and \ref{yedrinAB} are direct corollaries of Corollary 3.1 in \cite{Sch96}. Proposition \ref{yetterdrinfeld} comes from the same corollary in \cite{Sch96} and Proposition \ref{isom}.

\section{Construction of quantum algebras}\label{part2}
This section is devoted to the construction of three important quantum algebras: quantum groups, quantized Weyl algebras
and q-Boson algebras from the machinery built in the last section.
\subsection{Definitions and notations}
Assume that $q\in\mathbb{C}^*$ is not a root of unity. The q-numbers are defined by:
$$[n]=\frac{q^n-q^{-n}}{q-q^{-1}},\ \ [n]!=\prod_{i=1}^n[i],\ \ \bino{n}{k}=\frac{[n]!}{[k]![n-k]!}.$$
\indent
Let $\g$ be a finite dimensional complex semi-simple Lie algebra, $\mathfrak{h}$ its Cartan subalgebra,
$n=dim\mathfrak{h}=rank(\g)$, $Q$ its root lattice, $Q_+$ be the set of its positive roots,
$\mathcal{P}$ its weight lattice, $\Delta=\{\alpha_1\,\cdots,\alpha_n\}$ the set of simple roots,
$\mathcal{P}_+=\mathcal{P}\cap(\mathbb{Q}\ts_{\mathbb{Z}}Q)$. We denote $(\cdot,\cdot)$ the inner
product on $\mathfrak{h}^*$ given by the Killing form and $<\lambda,\mu>=\frac{2(\lambda,\mu)}{(\lambda,\lambda)}$
 for $\lambda,\mu\in\mathfrak{h}^*$. Define $q_i=q^{\frac{(\alpha_i,\alpha_i)}{2}}$.\\
\indent
Bricks of our construction are Hopf algebras $\widetilde{U_q^+}$ and $\widetilde{U_q^-}$,
which are defined by generators and relations:\\
(1). $\widetilde{U_q^+}$ is generated by $E_i$, ($i=1,\cdots,n$),
$K_\lambda^{\pm 1}$ ($\lambda\in\mathcal{P}_+$) with relations:\\
$$K_\lambda E_iK_\lambda^{-1}=q^{(\alpha_i,\lambda)}E_i,\ \
K_\lambda K_\lambda^{-1}=K_\lambda^{-1}K_\lambda=1.$$
It has a Hopf algebra structure given by:
$$\Delta(K_\lambda)=K_\lambda\ts K_\lambda,
\ \ \Delta(E_i)=E_i\ts K_{\alpha_i}^{-1}+1\ts E_i,$$
$$\ve(E_i)=0,\ \ \ve(K_\lambda)=1,\ \ S(E_i)=-E_iK_{\alpha_i},\ \ S(K_\lambda)=K_\lambda^{-1}.$$
(2). $\widetilde{U_q^-}$ is generated by $F_i$, ($i=1,\cdots,n$),
${K_\lambda'}^{\pm 1}$ ($\lambda\in\mathcal{P}_+$)  with relations:\\
$$K_\lambda'F_i{K_\lambda'}^{-1}=q^{-(\alpha_i,\lambda)}F_i,
\ \ K_\lambda'{K_\lambda'}^{-1}={K_\lambda'}^{-1}K_\lambda'=1.$$
It has a Hopf algebra structure given by:
$$\Delta(K_\lambda')=K_\lambda'\ts K_\lambda',\ \ \Delta(F_i)=F_i\ts 1+K_{\alpha_i}'\ts F_{\alpha_i},$$
$$\ve(F_i)=0,\ \ \ve(K_\lambda')=1,\ \ S(F_i)=-F_i{K_{\alpha_i}'}^{-1},\ \ S(K_\lambda')={K_\lambda'}^{-1}.$$
If $\lambda=\alpha_i$ for some $i$, we denote $K_i:=K_{\alpha_i}$.

\subsection{Construction of quantum groups}
The construction in this section can be found in \cite{KRT94}.\\
\indent
Denote by $D_\vp(\widetilde{U_q^+},\widetilde{U_q^-})$ the quantum double of $\widetilde{U_q^+}$ and $\widetilde{U_q^-}$, where
the generalized Hopf pairing $\vp:\widetilde{U_q^+}\times \widetilde{U_q^-}\ra\mathbb{C}$ is given by:
$$\vp(E_i,F_j)=\frac{\delta_{ij}}{q_i^{-1}-q_i},\ \ \vp(K_\lambda,K_\mu')=q^{-(\lambda,\mu)},
\ \ \vp(E_i,K_\lambda')=\vp(K_\lambda,F_i)=0,$$
$$\vp(E',1)=\ve(E'),\ \ \vp(1,F')=\ve(F'),\ \ \forall E'\in \widetilde{U_q^+},\ F'\in \widetilde{U_q^-}.$$
Now, from the definition of the multiplication in quantum double,
$$(1\ts F_j)(E_i\ts 1)=\vp(E_i,F_j)1\ts K_j'+E_i\ts F_i+\vp(S^{-1}(E_i),F_j)K_i^{-1}\ts 1,$$
it is to say,
$$E_iF_j-F_jE_i=\delta_{ij}\frac{K_i'-K_i^{-1}}{q_i-q_i^{-1}}.$$
With the same method, we also have:
$$K_\lambda'E_i=q^{(\lambda,\alpha_i)}E_iK_\lambda',\ \ F_iK_\lambda=q^{(\lambda,\alpha_i)}K_\lambda F_i.$$
\indent
So the quantum group $U_q(\g)$ can be obtained as follows: at first, to get a nondegenerate pairing, we
need to do the quotient by its left and right radical, denoted $I_l$ and $I_r$ respectively. Denote
$U_q^+=\widetilde{U_q^+}/I_l$ and $U_q^-=\widetilde{U_q^-}/I_r$. So $\vp$ induces a
nondegenerate pairing on
$U_q^+\ts U_q^-$, denote it also by $\vp$. The quantum group associated to $\g$ is just the quotient:
$$U_q(\g)=D_\vp(U_q^+,U_q^-)/(K_\lambda-K_\lambda').$$
\begin{remark}
 $I_l$ (resp. $I_r$) is generated by quantized Serre relations in $\widetilde{U_q^+}$
(resp. $\widetilde{U_q^-}$).
\end{remark}

\subsection{Heisenberg double and q-Boson algebras}
The procedure above, once applied to the Heisenberg double, will give q-Boson algebra.\\
\indent
In this section, we directly use $U_q^+$ and $U_q^-$ (it is to say, we add quantized Serre relations)
and generators $U_q^+=<e_i,t_\lambda^{\pm 1}>$, $U_q^-=<f_i,{t'}_\lambda^{\pm 1}>$
for making it distinct from the quantum double case. Moreover, $t_i:=t_{\alpha_i}$.\\
\indent
Now compute the multiplication structure between $U_q^+$ and $U_q^-$:\\
$$(1\sharp t_\lambda)(f_i\sharp 1)=\vp(t_\lambda,t_i')f_i\sharp t_\lambda
=q^{-(\alpha_i,\lambda)}f_i\sharp t_\lambda,$$
$$(1\sharp e_i)(t_\lambda'\sharp 1)=t_\lambda'\sharp e_i.$$
For this reason, it is better to adopt generators $e_i'=(q_i^{-1}-q_i)t_ie_i$ and this leads to:
$$\Delta(e_i')=e_i'\ts 1+t_i\ts e_i',\ \ t_\lambda e_i't_\lambda^{-1}=q^{(\alpha_i,\lambda)}e_i',\ \ \vp(e_i',f_j)=\delta_{ij}.$$
So at this time,
$$(1\sharp e_i')(t_\lambda'\sharp 1)=\vp(t_i,t_\lambda')t_\lambda'\sharp e_i'
=q^{-(\lambda,\alpha_i)}t_\lambda'\sharp e_i',$$
this is what we desired.\\
\indent
We calculate the relation between $e_i'$ and $f_j$:
$$(1\sharp e_i')(f_j\sharp 1)=\vp(e_i',f_j)+\vp(t_i,t_j')f_j\sharp e_i'
=q^{-(\alpha_i,\alpha_j)}f_j\sharp e_i'+\delta_{ij},$$
a simplification of the notation will give:
$$e_i'f_j=q^{-(\alpha_i,\alpha_j)}f_je_i'+\delta_{ij}.$$
Then all relations in q-Boson algebra have beed recovered and then
$$B_q(\g)\cong H_\vp(U_q^+,U_q^-)/(t_\lambda-t_\lambda'),$$
where $B_q(\g)$ is the q-Boson algebra defined in \cite{Nak94} and \cite{Nak04}.

\subsection{Action of quantum doubles on Heisenberg doubles}
Proposition \ref{yetterdrinfeld} gives an action of $D_\vp(U_q^+,U_q^-)$ on $H_\vp(U_q^+,U_q^-)$ such that
$H_\vp(U_q^+,U_q^-)$ is a $D_\vp(U_q^+,U_q^-)$-Yetter-Drinfel'd module. In this section, we will
show that, this action gives a $U_q(\g)$-module algebra structure on $H_\vp(U_q^+,U_q^-)$, but it
can not pass to the quotient to get an action on $B_q(\g)$. So more naturally, we need to introduce the
quantized Weyl algebra $W_q(\g)$: this is a subalgebra of $B_q(\g)$, a $U_q(\g)$-Yetter-Drinfel'd module and a $U_q(\g)$-module algebra.\\
\indent
At first, we calculate the action of $K_\lambda$ and $K_\lambda'$:
$$K_\lambda.e_i'=adK_\lambda(e_i')=q^{(\lambda,\alpha_i)}e_i',
\ \ K_\lambda'.e_i'=\vp(t_i,{K_\lambda'}^{-1})e_i'=q^{(\lambda,\alpha_i)}e_i',$$
$$K_\lambda.f_i=\vp(K_\lambda,t_i')f_i=q^{-(\lambda,\alpha_i)}f_i,
\ \ K_\lambda'.f_i=adK_\lambda'(f_i)=q^{-(\lambda,\alpha_i)}f_i.$$
So the action of $D_\vp(U_q^+,U_q^-)$ on $H_\vp(U_q^+,U_q^-)$ may pass to the quotient to give a $U_q(\g)$-module structure on $H_\vp(U_q^+,U_q^-)$.\\
\indent
But this in general can not give an action of $U_q(\g)$ on $B_q(\g)$ as we will show in an example later.
Denote $W_q(\g)$ the subalgebra of $B_q(\g)$ generated by $e_i'$ and $f_j$. It is a quantized version
of classical Weyl algebra: taking the Cartan matrix $C=0$ and $q=1$
will recover the usual Weyl algebra. (The condition $C=0$ has to do with quantized Serre relations.) The
name "quantized Weyl algebra" is proposed by A.Joseph in \cite{Jos95}. In \cite{Kas91}, M.Kashiwara calls it
"q-analogue of Boson".\\
\indent
From the definition of Schr\"{o}dinger representation and Proposition \ref{yetterdrinfeld}, we have:

\begin{proposition}
$W_q(\g)$ is a $U_q(\g)$-module algebra, moreover, it is a $U_q(\g)$-Yetter-Drinfel'd module.
\end{proposition}

\subsection{Example}
In this section, we compute the action of quantum double on Heisenberg double
in the $\g=\mathfrak{sl}_2$ case. Generators of $D_\vp(U_q^+,U_q^-)$ are $E,F,K^{\pm 1},{K'}^{\pm 1}$;
for $H_\vp(U_q^+,U_q^-)$, they are $e,f,t^{\pm 1},{t'}^{\pm 1}$.\\
\indent
At first, we calculate the action of $K$ and $K'$:
$$K.e'=adK(e')=q^2e',\ \ K'.e'=\vp(t,{K'}^{-1})e'=q^2e',$$
$$K.f=\vp(K,t')f=q^{-2}f,\ \ K'.f=adK'(f)=q^{-2}f.$$
So the action of $D_\vp(U_q^+,U_q^-)$ on $H_\vp(U_q^+,U_q^-)$ may pass to the quotient to give a $U_q(\mathfrak{sl}_2)$-module structure on $H_\vp(U_q^+,U_q^-)$.\\
\indent
But in general, it is not possible to obtain an action of $U_q(\mathfrak{sl}_2)$ on $B_q(\mathfrak{sl}_2)$ because of the following
computation:
$$E.t=adE(t)=(1-q^2)et^2,\ \ E.t'=0.$$
So it is natural to consider the action of $U_q(\g)$ on $W_q(\g)$.\\
\indent
In the end, it is better to write down all other actions: first recall that $e=\frac{t^{-1}e'}{q^{-1}-q}$,
suppose that $m\leq n$:
$$E^m.{e'}^n=\frac{[n+m-1]!}{[n-1]!}q^{-\frac{(2n+3+m)m}{2}}{e'}^{n+m},\ \
E^m.f^n=\frac{1}{(q^{-1}-q)^m}\frac{[n]!}{[n-m]!}q^{\frac{(2n-m-1)m}{2}}f^{n-m},$$
$$F^m.{e'}^n=(-1)^m\frac{[n]!}{[n-m]!}q^{\frac{(2n+3-m)m}{2}}{e'}^{n-m},\ \
F^m.f^n=\prod_{i=0}^{m-1}(1-q^{-2(n+i)})f^{n+m}.$$
These formulas will be useful for the calculation in Section \ref{wqbraid}.

\section{Modules over q-Boson algebras}\label{part3}
Sometimes, we use notations $U_q$, $B_q$ and $W_q$ instead of $U_q(\g)$, $B_q(\g)$ and $W_q(\g)$.
Capital letters will be used for elements in $U_q(\g)$, lowercases for $B_q(\g)$ and $W_q(\g)$.\\
\indent
Denote $B_q^{++}$ and $B_q^{--}$ subalgebras of $B_q(\g)$
generated by $e_i'$, $f_j$ ($1\leq i,j\leq n$) respectively and $B_q^0$ the subalgebra generated by
$t_\lambda^{\pm 1}$ ($\lambda\in\mathcal{P}_+$). Let $U_q^0$ be the sub-Hopf algebra of $U_q$ generated
by $K_\lambda^{\pm 1}$ ($\lambda\in\mathcal{P}_+$).

\subsection{Construction of $W_q(\g)$ from braiding}\label{wqbraid}
\indent
We have seen in the previous section that $W_q(\g)$ is in
${}^{U_q}_{U_q}\mathcal{YD}$. \\
\indent
On $B_q^{++}$, there exists a $U_q$-Yetter-Drinfel'd module algebra structure: the $U_q$-module structure is given by the Schr\"{o}dinger representation and the $U_q$-comodule
structure is given by $\delta(e_i')=(q_i^{-1}-q_i)K_iE_i\ts 1+K_i\ts e_i'$. It is easy to see that $B_q^{++}$ is indeed a $U_q$-module because the adjoint action preserves $B_q^{++}$. These structures are compatible because
$\delta$ is just $\Delta$ in $U_q$.\\

\indent
In the category ${}^{U_q}_{U_q}\mathcal{YD}$, we can use the
braiding to give the tensor product of two module algebras a structure of algebra. For our purpose,
consider
$W_q\ts W_q$, and denote the braiding by $\sigma$;
then $(m\ts m)\circ(id\ts\sigma\ts id)$:
$$W_q\ts W_q\ts W_q\ts W_q\ra W_q\ts W_q\ts W_q\ts
W_q\ra W_q\ts W_q,$$ gives $W_q\ts W_q$ a structure of
algebra. We denote this algebra by
$W_q\underline{\ts} W_q$. \\
\indent
We want to restrict this
braiding to the subspace $B_q^{--}\ts B_q^{++}\subset W_q\ts W_q$. This requires to restrict the $U_q$-comodule structure on $W_q$ to $B_q^{++}$ and the $U_q$-module structure on $W_q$ to $B_q^{--}$. The comodule structure
could be directly restricted as we did
in the beginning of this section; the possibility for the restriction of the module structure comes from the fact that the
Schr\"{o}dinger representation makes $B_q^{--}$ stable. This gives an algebra $B_q^{--}\underline{\ts} B_q^{++}$. \\
\indent
We calculate the product in $B_q^{--}\underline{\ts} B_q^{++}$:\\
$$(f_i\ts 1)(1\ts e_j')=f_i\ts e_j',$$
\begin{eqnarray*}
(1\ts e_i')(f_j\ts 1) &=& \sum (e_i')_{(-1)}\cdot f_j\ts (e_i')_{(0)}\\
&=& ((q^{-1}-q)K_iE_i)\cdot f_j\ts 1+K_i\cdot f_j\ts e_i'\\
&=& \delta_{ij}+q^{-(\alpha_i,\alpha_j)}f_j\ts e_i'.
\end{eqnarray*}
This is nothing but relations in quantized Weyl algebra $W_q(\g)$.
Moreover, as a vector space, $W_q(\g)$ has a decomposition $W_q(\g)\cong B_q^{--}\ts B_q^{++}$,
and the inverse map is given by the multiplication; so we have:
\begin{proposition}
There exists an algebra isomorphism:
$$B_q^{--}\underline{\ts} B_q^{++}\cong W_q(\g),\ \ f\ts e\mapsto fe,$$
where $f\in B_q^{--}$, $e\in B_q^{++}$.
\end{proposition}

\subsection{Modules over $W_q(\g)$}
This subsection is devoted to studying modules over $W_q(\g)$ with finiteness condition.\\
\indent
We define the category $\mathcal{O}(W_q)$ as a full subcategory of $W_q(\g)$-modules which contains those $W_q(\g)$-modules satisfying: for any $M$ in $\mathcal{O}(W_q)$ and any $m\in M$, there exists an integer $l>0$ such that for any $1\leq i_1,\cdots,i_l\leq n$, $e_{i_1}'e_{i_2}'\cdots e_{i_l}'.m=0$.\\
\indent
Let $M$ be a $W_q(\g)$-module in $\mathcal{O}(W_q)$.  The braided Hopf algebras $B_q^{++}$ and $B_q^{--}$ are both $\mathbb{N}$-graded by defining $deg(e_i')=deg(f_i)=1$. We denote $B_q^{++}(n)$ the finite dimensional subspace of $B_q^{++}$ containing elements of degree $n$ and $(B_q^{++})^g=\bigoplus_{n\geq 0}B_q^{++}(n)^*$ the graded dual coalgebra of $B_q^{++}$.\\
\indent
Recall that there exists a pairing between $B_q^{++}$ and $B_q^{--}$ given by
$\vp(e_i',f_j)=\delta_{ij}$. Because $\vp(B_q^{++}(n), B_q^{--}(m))=0$ for $m\neq n$ and the restriction of $\vp$ on $B_q^{++}(n)\times B_q^{--}(n)$, $n\geq 0$, is non-degenerate, the graded dual of $B_q^{++}$ is anti-isomorphic to $B_q^{--}$ as graded coalgebras. The prefix "anti" comes from Lemma 2.5 in \cite{Mas09} for the restricted pairing $B_q^{++}\times B_q^{--}\ra \mathbb{C}$ on braided Hopf algebras. Thus we obtain
$$(B_q^{++})^g\cong B_q^{--}.$$

From the definition, the action of $B_q^{++}$ on $M$ is locally nilpotent, and so from duality, we obtain a left
$(B_q^{++})^g$-comodule structure on $M$. With the help of the isomorphism above, there is a left
$B_q^{--}$-comodule structure on $M$ given by: if we denote $\rho:M\ra B_q^{--}\ts M$ by
$\rho(m)=\sum m_{(-1)}\ts m_{(0)}$, then for $e\in B_q^{++}$,
$$e.m=\sum \vp(e,m_{(-1)})m_{(0)}.$$
\\
\indent
Thus from a $W_q$-module, we obtain a $B_q^{--}$-module which is at the same time a $B_q^{--}$-comodule, and is, moreover, a braided Hopf module.\\

\begin{remark} \label{remark7}
Here, for the left $B_q^{--}$-comodule structure on $M$, it is needed to consider the braided Hopf algebra structure
on $B_q^{--}$, it is to say, $\Delta_0:B_q^{--}\ra B_q^{--}\underline{\ts} B_q^{--}$,
$\Delta_0(f_i)=f_i\ts 1+1\ts f_i$. We use the primitive coproduct but twist the multiplication
structure to get a good duality between left $B_q^{++}$-modules and right $B_q^{--}$-comodules. For
the left $B_q^{--}$-module structure on $M$, we keep the ordinary coproduct $\Delta(f_i)=f_i\ts 1+t_i\ts f_i\in B_q^-\ts B_q^{--}$.
\end{remark}
\indent
We define a linear projection
$\pi:B_q^-\ra B_q^{--}$ by $ft\mapsto f\ve(t)$, where $f\in B_q^{--}$ and $t\in B_q^0$.
\begin{proposition}\label{compatibility}
The following compatibility relation between the module and comodule structures defined above holds: for $f\in B_q^{--}$, $m\in M$,
\begin{eqnarray}\label{projection}
\rho(f.m)=\sum \pi(f_{(1)}m_{(-1)})\ts f_{(2)}.m_{(0)}=\Delta_0(f)\rho(m).
\end{eqnarray}
\end{proposition}
\begin{proof}
At first we calculate $\rho(f.m)$: for any $e\in B_q^{++}$,
\begin{eqnarray*}
e(f.m)=(ef).m &=& \sum (e_{(-1)}\cdot f)e_{(0)}m\\
&=& \sum \vp(e_{(0)},m_{(-1)})(e_{(-1)}\cdot f)m_{(0)}\\
&=& \sum \vp(e_{(-1)},f_{(1)})\vp(e_{(0)},m_{(-1)})f_{(2)}.m_{(0)}\\
&=& \sum \vp(e,f_{(1)}m_{(-1)})f_{(2)}m_{(0)}.
\end{eqnarray*}
Here, $f_{(1)}m_{(-1)}$ is not necessary in $B_q^{--}$, but we always have
$$\vp(e,f_{(1)}m_{(-1)})=\vp(e,\pi(f_{(1)}m_{(-1)})),$$
which gives the first equality.\\
\indent
For the second one, from the definition of braiding,
\begin{equation}\label{prim}
\Delta_0(f)\rho(m)=\sum f^{(1)}((f^{(2)})_{(-1)}.m_{(-1)})\ts (f^{(2)})_{(0)}m_{(0)},
\end{equation}
where $\Delta_0(f)=\sum f^{(1)}\ts f^{(2)}$.\\
\indent
As said in Remark \ref{remark7}, we look $B_q^{--}$ as a braided Hopf algebra when considering comodule structure, so $(f^{(2)})_{(0)}=f^{(2)}$ and from the definition of $\Delta$ in $B_q^{--}$, 
$$\pi(f_{(1)})((f_{(2)})_{(-1)}.m_{(-1)})=\pi(f_{(1)}m_{(-1)}).$$
Notice that $\Delta_0(f)=(\pi\ts id)(\Delta(f))$. So
$$\sum f^{(1)}\ts f^{(2)}=\sum\pi(f_{(1)})\ts f_{(2)}$$
and the formula above gives the second equality.
\end{proof}
In the categorical language, the proposition above says:
\begin{corollary}\label{equiv}
There exists an equivalence of category $\mathcal{O}(W_q)\sim {}_{B_q^{--}}^{B_q^{--}}\mathcal{M}$.
\end{corollary}
The following theorem gives the structural result for $W_q$-modules with finiteness condition.
\begin{theorem}\label{Wq}
There exists an equivalence of category $\mathcal{O}(W_q)\sim Vect$, where $Vect$ is the category of vector spaces. The equivalence is given by:
$$M\mapsto M^{co\rho},\ \ V\mapsto B_q^{--}\ts V,$$
where $M\in\mathcal{O}(W_q)$, $V\in Vect$, $M^{co\rho}=\{m\in M|\ \rho(m)=1\ts m\}$ is the set of coinvariants.
\end{theorem}
\begin{proof}
We have seen in the Corollary \ref{equiv} that $\mathcal{O}(W_q)$ and ${}_{B_q^{--}}^{B_q^{--}}\mathcal{M}$ are equivalent, so the theorem comes from the triviality of braided Hopf modules as showed in Proposition \ref{triviality}.
\end{proof}

It is better to write down an explicit formula for $\rho$.\\
\indent
For $\beta\in Q_+\cup\{0\}$, define
$$(B_q^{++})_\beta=\{x\in B_q^{++}|\ t_\lambda xt_\lambda^{-1}=q^{(\beta,\lambda)}x, \forall t_\lambda\in U_q^0\}.$$
Moreover, $(B_q^{--})_{-\beta}$ can be similarly defined. Elements in $(B_q^{++})_\beta$ (resp, $(B_q^{--})_{-\beta}$) are
called of degree $\beta$ (resp, $-\beta$).
Let $e_{\alpha,i}\in (B_q^{++})_{\alpha}$, $1\leq i\leq dim((B_q^{++})_{\alpha})$ be a basis of
$B_q^{++}$, $f_{\beta,j}$ be the dual basis respected to $\vp$, such that
$$\vp(e_{\alpha,i},f_{\beta,j})=\delta_{ij}\delta_{\alpha\beta}.$$
Formally, define $\rr=\sum_{i,\alpha}f_{\alpha,i}\ts e_{\alpha,i}$. Because $M\in\mathcal{O}(B_q)$,
$\rr(1\ts m)$ is well-defined for every $m\in M$.
\begin{proposition}\label{rho}
For $m\in M$, $\rho(m)=\rr(1\ts m)$.
\end{proposition}
\begin{proof}
Let $m\in M$, we can write $\rho(m)=\sum_{\alpha,j} f_{\alpha,j}\ts m_{\alpha,j},$
then from the definition of left comodule structure,
$$e_{\beta,i}.m=\sum_{\alpha,j}\vp(e_{\beta,i},f_{\alpha,j})m_{\alpha,j}=m_{\beta,i}.$$
So
$\rho(m)=\sum_{\alpha,i} f_{\alpha,i}\ts e_{\alpha,i}.m=\rr(1\ts m).$
\end{proof}
\indent
It is better to verify formula (\ref{projection}) in an example.\\
\noindent
\textbf{Example.} Consider the $\mathfrak{sl}_2$ case, generators of $B_q(\mathfrak{sl}_2)$ will be
denoted by $e,f,t^{\pm 1}$.
We choose $m\in M$ such that $e.m\neq 0$, $e^2.m=0$.
Then $e^2f=q^{-4}fe^2+(q^{-2}+1)e$,
$$\rho(f.m)=1\ts fm+ f\ts efm+f^2\ts \frac{1}{q^{-2}+1}e^2fm
=1\ts fm+f\ts q^{-2}fem+f\ts m+ f^2\ts em,$$
\begin{eqnarray*}
(\pi\ts id)(\Delta(f)\rho(m))&=&(\pi\ts id)(f\ts m+f^2\ts em+t\ts fm+tf\ts fem)\\
&=& f\ts m+f^2\ts em+1\ts fm+f\ts q^{-2}fem.
\end{eqnarray*}
For the primitive coproduct,
\begin{eqnarray*}
\Delta_0(f)\rho(m) &=& f\ts m+f^2\ts em+1\ts fm+q^{-2}f\ts fem.
\end{eqnarray*}
\qed
\\

\indent
For a $W_q$-module $M$, $0\neq m\in M$ is called a maximal vector if it is annihilated by all $e_i'$. The set of all maximal vectors in $M$ is
denoted by $K(M)$.\\
\indent
The following lemma is a direct consequence of the definition.
\begin{lemma}\label{action}
Suppose that $m\in M^{co\rho}$. Then for any non-constant element $e\in B_q^{++}$, $e.m=0$.
\end{lemma}
\begin{lemma}\label{pairing}
Let $f\in B_q^{--}$, $f\notin \mathbb{C}^*$, such that for any $i$, $e_i'.f=0$. Then $f=0$. 
\end{lemma}
\begin{proof}
If $e_i'.f=0$ for any $i$, $f$ is annihilated by all non-constant elements in $B_q^{++}$. For any $e\in B_q^{++}$, $e.f=\sum \vp(e,f_{(1)})f_{(2)}$, we can suppose that these $f_{(2)}$ are linearly independent, so $\vp(e,f_{(1)})=0$ for any $f_{(1)}$ and any non-constant $e\in B_q^{++}$, the non-degeneracy of the Hopf pairing forces $f_{(1)}$ to be constants.\\
\indent
So now $f=(id\ts\ve)\Delta(f)=\sum f_{(1)}\ve(f_{(2)})\in\mathbb{C}$ and it must be $0$ from the hypothesis.
\end{proof}

Combined with Theorem \ref{Wq} above, Lemma \ref{action} and \ref{pairing} give:
\begin{corollary}\label{K(M)}
Let $M\in\mathcal{O}(W_q)$ be a $W_q$-module. Then $M^{co\rho}=K(M)$.
\end{corollary}
\begin{remark}
The corollary above means that coinvariants are exactly those "extremal vectors" defined in \cite{Nak04}.
\end{remark}

\subsection{Modules over $B_q(\g)$}
Recall the definition of the category $\mathcal{O}(B_q)$ in \cite{Nak04}: it is a full subcategory of left module category over $B_q(\g)$
containing objects satisfying the following conditions:\\
(i). Any object $M$ in $\mathcal{O}(B_q)$ has a weight space decomposition:
$$M=\bigoplus_{\lambda\in\mathcal{P}}M_\lambda,\ \ M_\lambda=\{m\in M|\ t_\mu.m=q^{(\mu,\lambda)}m\}.$$
(ii). For any $M$ in $\mathcal{O}(B_q)$ and any $m\in M$, there exists an integer $l>0$ such that for any $1\leq i_1,\cdots,i_l\leq n$, $e_{i_1}'e_{i_2}'\cdots e_{i_l}'.m=0$.\\
\indent
Moreover, we denote $\mathcal{O}'(B_q)$ the category of $B_q$-modules satisfying only (ii) above.\\

\indent
The main theorem of this article is the following structural result.
\begin{theorem}\label{main}
There exists an equivalence of category $\mathcal{O}'(B_q)\sim{}_{U_q^0}Mod$. The equivalence is given by:
$$M\mapsto K(M),\ \ \ V\mapsto B_q^{--}\ts V,$$
where $M\in \mathcal{O}'(B_q)$, $V$ is a $U_q^0$-module and $K(M)$ is the set of maximal vectors in $M$, when looked as a $W_q$-module.\\
\indent 
Moreover, when restricted to the subcategory $\mathcal{O}(B_q)$, the equivalence above gives $\mathcal{O}(B_q)\sim {}_{\mathcal{P}}Gr$, where the latter is the category of $\mathcal{P}$-graded vector spaces.
\end{theorem}
\begin{remark}
It should be remarked that the method in this article can be generalized to Nichols algebras of diagonal type as in \cite{Mas09}.
\end{remark}
The next subsection is devoted to the proof of this result.

\subsection{Proof of Theorem \ref{main}}
We proceed to the proof of Theorem \ref{main}.\\
\indent
At first, for any $U_q^0$-module $V$, we may look it as a vector space through the forgetful functor. So from Theorem \ref{Wq}, $N=B_q^{--}\ts V$ admits a locally finite $W_q$-module structure such that $N^{co\rho}=V$. Moreover, if the $U_q^0$-module structure on $V$ is under consideration, there exists a $B_q$-module structure over $B_q^{--}\ts V$ given by: 
for $v\in V$, $x,f\in B_q^{--}$, $e\in B_q^{++}$, $t\in B_q^0$,
$$e.(x\ts v)=\sum \vp(e,x_{(1)})x_{(2)}\ts v,\ \ f.(x\ts v)=fx\ts v,\ \ t.(x\ts v)=txt^{-1}\ts tv.$$
\indent
As a summary, the discussion above gives a functor ${}_{U_q^0}Mod\ra \mathcal{O}'(B_q)$.\\
\indent
From now on, let $M\in \mathcal{O}'(B_q)$ be a $B_q$-module with finiteness condition.\\
\indent
The restriction from $B_q$-modules to $W_q$-modules gives a functor $\mathcal{O}'(B_q)\ra \mathcal{O}(W_q)$, thus we obtain a functor $\mathcal{O}'(B_q)\ra {}_{U_q^0}Mod$ by composing with the equivalence functor $\mathcal{O}(W_q)\ra Vect$.\\
\indent
From Theorem \ref{Wq} and the module structures defined above, these two functors give an equivalence of category $\mathcal{O}'(B_q)\sim {}_{U_q^0}Mod$. So the first point of Theorem \ref{main} comes from Corollary \ref{K(M)}.\\
\indent
The second point comes from the equivalence of the $U_q^0$-modules satisfying condition (i) in $\mathcal{O}(B_q)$ and $\mathcal{P}$-graded vector spaces.

\subsection{Semi-simplicity}
Now it is easy to obtain the structural result for $\mathcal{O}(B_q)$ as showed in the article of Nakashima.\\
\indent
The following result is a direct corollary of Theorem \ref{main} and Lemma \ref{action}.
\begin{corollary}
Let $M\in\mathcal{O}(B_q)$ be a nontrivial $B_q$-module. There exists nonzero maximal vectors in $M$.
\end{corollary}
\indent
For $\lambda\in\mathcal{P}$, define a left ideal of $B_q$ by:
$$I_\lambda=\sum_i B_q e_i'+\sum B_q(t_\alpha-q^{(\lambda,\alpha)}),$$
denote $H(\lambda)=B_q/I_\lambda$, then $H(\lambda)$ is a free $B_q^{--}$-module of rank 1, generated by $1$. The following structural results follow from Theorem \ref{main}.
\begin{corollary}
Let $M\in\mathcal{O}(B_q)$, $v\in M$ be a maximal vector of weight $\lambda$. Then
$B_q^{--}\ts \mathbb{C}v\ra H(\lambda)$, $F\ts v\mapsto F.v$ is an isomorphism of $B_q$-modules.
In particular, $H(\lambda)$ are all simple $B_q$-modules.
\end{corollary}
\begin{corollary}
(1). Let $M$ be a simple $B_q$-module. Then there exists some $\lambda$ such that $M\cong H(\lambda)$.\\
(2). Suppose that $M\in\mathcal{O}(B_q)$. Then $M$ is semi-simple.\\
\end{corollary}

\subsection{Extremal projector}
This section is devoted to the computation of the projection $P$ in the $\mathfrak{sl}_2$ case, and show that
it is exactly the operator given by the formula (3.2.2) of \cite{Kas91}. Moreover, changing $q$ by $q^{-1}$
will lead us to the formula in Example 5.1 of \cite{Nak04}.\\
\indent
At first, we calculate $\Delta_0:B_q^{--}\ra B_q^{--}\underline{\ts} B_q^{--}$ and the antipode; note that
the multiplication is twisted in the right hand side.
\begin{lemma}
(1). $\D\Delta_0(f^n)=\sum_{p=0}^n \bino{n}{p} q^{p^2-np}f^p\ts f^{n-p}$,\\
(2). $S(f^n)=(-1)^nq^{-n(n-1)}f^n$.
\end{lemma}
\begin{proof}
(1). By induction on $n$ and use the identity
$$\bino{n+1}{p+1}=q^{p+1}\bino{n}{p+1}+q^{p-n}\bino{n}{p}.$$
(2). Apply $S\ts id$ on the formula of $\Delta_0(f^n)$ then use induction and the identity
$$\sum_{i=0}^r (-1)^iq^{-i(r-1)}\bino{r}{i}=0.$$
\end{proof}

From Proposition \ref{rho}, in the $\mathfrak{sl}_2$ case, we obtain a well-defined morphism:
$$\rho(m)=\sum_{n=0}^\infty q^{\frac{n(n-1)}{2}}\frac{f^n}{[n]!}\ts e^n.m,$$
and then
$$P(m)=\sum_{n=0}^\infty (-1)^n q^{-\frac{n(n-1)}{2}}\frac{f^n}{[n]!}e^n.m.$$
It is exactly the operator defined in \cite{Kas91}, (3.2.2) and almost
the extremal projector $\Gamma$ in \cite{Nak04}.\\
\\


\begin{thebibliography}{99}

\bibitem{AS94}
  N. Andruskiewitsch, H.-J. Schneider,
  \emph{Pointed Hopf algebras}.
  in: New directions in Hopf algebra theory, in: Math. Sci. Res. Inst. Publ. Vol. 43, Cambridge
  Univ. Press, Cambridge, 2002, pp. 1-68.

\bibitem{DT94}
    Y.Doi, M. Takeuchi,
    \emph{Multiplication alteration by two-cocycles-the quantum version}.
    Comm. Algebra, \textbf{22} (1994), no.14, 5715-5732.
\bibitem{JR09}
  R.-Q. Jian, M. Rosso,
  \emph{Quantum $B_\infty$-algebras I: constructions of YB-algebras and YB-coalgebras}.
  arXiv: 0904.2964

\bibitem{Jos95}
  A. Joseph,
  \emph{Quantum groups and their primitive ideals}.
  Ergebnisse der Mathematik und ihrer Grenzgebiete 29, Springer-Verlag, 1995.


\bibitem{Lu94}
  J.-H. Lu,
  \emph{On the Drinfel'd double and the Heisenberg double of a Hopf algebra}.
  Duke Math.J,
  \textbf{74(3)},
  763-776(1994).

\bibitem{Kas91}
  M. Kashiwara,
  \emph{On crystal bases of the q-analogue of universal enveloping algebras}.
  Duke Math.J.,
  \textbf{63},
  465-516(1991).

\bibitem{KRT94}
  C. Kassel, M. Rosso, V. Turaev,
  \emph{Quantum groups and knot invariants}.
  Panoramas et Synth\`{e}ses, No.5, Soci\'{e}t\'{e} Math\'{e}matique de France, 1997.

\bibitem{Maj95}
  S. Majid,
  \emph{Foundations of quantum group theory}.
  Cambridge Univ. Press, 1995.

\bibitem{Mas09}
   A. Masuoka,
   \emph{Generalized q-Boson algebras and their integrable modules}.
   J. Algebra, \textbf{322} (2009), 2199-2219.
   
\bibitem{MS99}
   S. Montgomery, H.-J. Schneider,
   \emph{Prime ideals in Hopf Galois extensions}.
   Israel J. Math, \textbf{112} (1999), 187-235.

\bibitem{Nak94}
  T. Nakashima,
  \emph{Quantum R-matrix and intertwiners for the Kashiwara algebras}.
  Commun. Math. Phys,
  \textbf{164},
  239-258(1994).

\bibitem{Nak04}
  T. Nakashima,
  \emph{Extremal projectors of q-Boson algebras}.
  Commun. Math. Phys,
  \textbf{244},
  285-296(2004).

\bibitem{Rad85}
  D. Radford,
  \emph{The structure of Hopf algebras with a projection}.
  J. Algebra,
  \textbf{92},
  322-347(1985).

\bibitem{Ros89}
  M. Rosso,
  \emph{An analogue of P.B.W theorem and the universal R-matrix for $U_h(sl(N+1))$}.
  Comm. Math. Phys,
  \textbf{124},
  307-318(1989)

\bibitem{Sem10}
  A.M. Semikhatov,
  \emph{Heisenberg double $H(B^*)$ as a braided commutative Yetter-Drinfeld
  module algebra over the Drinfeld double}.
  arXiv:1001.0733.
  
\bibitem{Sch96}
   P. Schauenberg,
   \emph{Hopf bimodules over Hopf-Galois extensions, Miyashita-Ulbrich actions, and monoidal center constructions}.
    Comm. Algebra, \textbf{24} (1996), 143-163.
   
  
\bibitem{Swe94}
  M.E. Sweedler,
  \emph{Hopf algebras}.
  W.A.Benjamin, New York, 1969.

\bibitem{Tan92}
  T. Tanisaki,
  \emph{Killing forms, Harish-Chandra isomorphisms, and universal R-matrices for quantum algebras}.
  International Journal of Modern Physics A, Vol.7, Suppl. 1B,
  941-961(1992)

\end{thebibliography}
\end{document}